\documentclass[10pt,a4paper]{amsart}
\usepackage{amsfonts, amsmath, amssymb, amsthm, amscd, hyperref}
\usepackage{url,paralist}
\usepackage{anysize}
\usepackage{diagrams}

\newtheorem{thm}{Theorem}
\newtheorem*{thm*}{Theorem}

\newtheorem*{claim*}{Claim}

\newtheorem{cor}[thm]{Corollary}
\newtheorem{proposition}[thm]{Proposition}

\theoremstyle{definition}

\theoremstyle{remark}
\newtheorem*{rem}{Remark}

\renewcommand{\int}{\mathop{\rm int}}
\renewcommand{\epsilon}{\varepsilon}

\newcommand*{\R}{\mathbb{R}}                                   
\newcommand*{\Z}{\mathbb{Z}}                                   

\begin{document}
\title{Extensions of theorems of Rattray and Makeev}
\author{Pavle~V.~M.~Blagojevi\'c$^{*}$}
\email{pavleb@mi.sanu.ac.rs}
\address{Pavle~V.~M.~Blagojevi\'c, Mathemati\v cki Institut SANU, Knez Michailova 36, 11001 Beograd,
Serbia}
\thanks{$^{*}$The research leading to these results has received funding from the European Research
Council under the European Union's Seventh Framework Programme (FP7/2007-2013) /
ERC Grant agreement no.~247029-SDModels. Also supported by the grant ON 174008 of the Serbian
Ministry of Education and Science.}
\author{Roman~Karasev$^{**}$}
\thanks{$^{**}$The research of R.N.~Karasev is supported by the Dynasty Foundation, the President's of Russian Federation grant MD-352.2012.1, the Russian Foundation for Basic Research grants 10-01-00096 and 10-01-00139, the Federal Program ``Scientific and scientific-pedagogical staff of innovative Russia'' 2009--2013, and the Russian government project 11.G34.31.0053.}
\email{r\_n\_karasev@mail.ru}
\address{Roman~Karasev, Departement of Mathematics, Moscow Institute of Physics and Technology, Institutskiy per. 9, Dolgoprudny, Russia 141700}
\address{Roman~Karasev, Laboratory of Discrete and Computational Geometry, Yaroslavl' State University, Sovetskaya st. 14, Yaroslavl', Russia 150000}

\begin{abstract}
We consider extensions of the Rattray theorem and two Makeev's theorems, showing that they hold for several maps, measures, or functions simultaneously, 
when we consider orthonormal $k$-frames in $\R^n$ instead of orthonormal basis (full frames).

We also present new results on simultaneous partition of several measures into parts by $k$ mutually orthogonal hyperplanes.

In the case $k=2$ we relate the Rattray and Makeev type results with the well known embedding problem for projective spaces. 
\end{abstract}

\maketitle

\section{Introduction}

\noindent In this paper we consider extensions of the following results of
Rattray and Makeev:

\begin{compactitem}
\item  any odd continuous map $S^{n-1}\rightarrow S^{n-1}$ maps some orthonormal basis to an orthonormal
basis, the Rattray theorem \cite{ratt1954};

\item for any absolutely continuous probabilistic measure $\mu $ in $\R^{n}$ there exist $n$
mutually orthogonal hyperplanes $H_{1},\ldots ,H_{n}$ such that any two of
them partition $\mu $ into $4$ equal parts, the Makeev theorem \cite[Theorem~4]{mak2007-1}.

\end{compactitem}

\noindent These results share a common family of possible solutions, the manifold of all orthonormal basis $\mathrm{O}(n)$ in $\mathbb{R}^{n}$. Moreover, they can be seen as a consequence of a single result, Theorem~\ref{rattray}, proved implicitly already in~\cite{ratt1954}.

\medskip

A continuous function $f:S^{n-1}\times S^{n-1}\rightarrow \mathbb{R}$ will
be called

\begin{compactenum}[\rm(a)]
\item \emph{odd}, if for any $x,y\in S^{n-1}$
\begin{equation*}
f(-x,y)=-f(x,y),\ f(x,-y)=-f(x,y);
\end{equation*}

\item \emph{symmetric}, if for any $x,y\in S^{n-1}$
\begin{equation*}
f(x,y)=f(y,x).
\end{equation*}
\end{compactenum}

\begin{thm}
\label{rattray} Suppose $f:S^{n-1}\times S^{n-1}\rightarrow \mathbb{R}$ is
an odd and symmetric function. Then there exists an orthonormal basis $%
(e_{1},\ldots ,e_{n})\in \mathrm{O}(n)$ such that for any $i<j$
\begin{equation*}
f(e_{i},e_{j})=0.
\end{equation*}
\end{thm}

\begin{proof}
Consider a particular case when $f(x,y)$ is a generic symmetric bilinear form. It follows from the diagonalization theorem in linear algebra that the required orthonormal basis $e_{1},\ldots ,e_{n}$ exists and is unique modulo the action of the group $W_{n}=(\mathbb{Z}_{2})^{n}\rtimes \Sigma_{n}\subset \mathrm{O}(n)$. Here the group $W_{n}$ acts on basis $\left(e_{1},\ldots ,e_{n}\right) \in \mathrm{O}(n)$ by
\begin{equation*}
\varepsilon _{i}\cdot \left( e_{1},\ldots ,e_{n}\right) =\left(
e_{1}^{\prime },\ldots ,e_{n}^{\prime }\right) \text{ where }e_{j}^{\prime
}=\left\{
\begin{array}{rrr}
-e_{j} &  & \text{, for }j=i \\
e_{j} &  & \text{, for }j\neq i%
\end{array}%
\right.
\end{equation*}%
for the generators $\varepsilon _{1},\ldots ,\varepsilon _{n}$ of the component $%
(\mathbb{Z}_{2})^{n}$ and by%
\begin{equation*}
\pi \cdot \left( e_{1},\ldots ,e_{n}\right) =\left( e_{\pi (1)},\ldots
,e_{\pi (n)}\right)
\end{equation*}%
for the permutation $\pi \in \Sigma _{n}$ from the symmetric group component
of the group $W_{n}$.

\noindent Let us show that:

\begin{compactitem}
\item the differential of the corresponding system of equations evaluated at
the solution $e_{1},\ldots ,e_{n}$ is nonzero, and

\item the solution set represents a nonzero element of the $0$-homology $H_{0}(%
\mathrm{O}(n)/W_{n};\mathbb{F}_{2})$.
\end{compactitem}

\noindent Suppose the base vector $e_{i}$ has coordinates $b_{ij}$, and
\begin{equation*}
f(x,y)=\sum_{i}\lambda _{i}x_{i}y_{i}
\end{equation*}%
in the coordinate representation. Since $f$ is a generic symmetric bilinear form we can assume that $\lambda _{1},\ldots ,\lambda _{n}$ are distinct real numbers. The solution is $b_{ij}=\delta _{ij}$, and its first order deformation is $b_{ij}=\delta _{ij}+s_{ij}$, where $s_{ij}$ is a skew symmetric $n\times n$ matrix. Consider
\begin{equation*}
f(e_{k},e_{l})=\sum_{i}\lambda _{i}b_{ik}b_{il}.
\end{equation*}%
The linear part, with respect to $s_{ij}$, is
\begin{equation*}
df(e_{k},e_{l})=\sum_{i}\lambda _{i}\delta _{ik}s_{il}+\sum_{i}\lambda
_{i}s_{ik}\delta _{il}=\lambda _{k}s_{kl}+\lambda _{l}s_{lk}=(\lambda
_{k}-\lambda _{l})s_{kl}.
\end{equation*}%
Since all values $\lambda _{k}-\lambda _{l}$ are nonzero, that the differentials $df(e_{k},e_{l})$ give together a bijective map from the space of skew symmetric matrices to the space of all symmetric expressions of the form $t_{kl}$ for $k\neq l$. 

Since any $f$ can be $W_n$-deformed (by a convex combination) to this
particular case, it follows that for generic $f$ the solution set represents the generator of $H_{0}(\mathrm{O}(n)/W_{n};\mathbb{F}_{2})$ (and is nonempty). Therefore, the solution set must be nonempty for all other $f$ by compactness considerations.
\end{proof}

\medskip

\noindent In this paper we consider the following generalized problems of
Rattray and Makeev type.

\subsection*{Generalized Rattray problem}

\noindent Determine the set $\mathcal{R}_{odd}^{orth}\subset \mathbb{N}^{3}$
[$\mathcal{R}_{odd,sym}^{orth}\subset \mathbb{N}^{3}$] of all triples $%
(n,m,k)$ with the property that for any collection $f_{1},\ldots ,f_{m}$ of $%
m$ odd [and symmetric] functions $S^{n-1}\times S^{n-1}\rightarrow \mathbb{R}
$ there exists an orthonormal $k$-frame $\left( e_{1},\ldots ,e_{k}\right)
\in V_{n}^{k}$ such that for any $1\leq l\leq m$ and $1\leq i<j\leq k$
\begin{equation*}
f_{l}(e_{i},e_{j})=0.
\end{equation*}%
Here $V_{n}^{k}$ stands for the Stiefel manifold of all orthonormal $k$%
-frames in $\mathbb{R}^{n}$.

\noindent This problem has a natural variation when the requirement for the
vectors $e_{1},\ldots ,e_{k}$ to be orthonormal is dropped. Determine the
set $\mathcal{R}_{odd}\subset \mathbb{N}^{3}$ [$\mathcal{R}_{odd,sym}\subset
\mathbb{N}^{3}$] off all triples $(n,m,k)$ with the property that for any
collection $f_{1},\ldots ,f_{m}$ of $m$ odd [and symmetric] functions $%
S^{n-1}\times S^{n-1}\rightarrow \mathbb{R}$ there exist $k$ unit vectors $%
e_{1},\ldots ,e_{k}$ such that for any $1\leq l\leq m$ and $1\leq i<j\leq k$
\begin{equation*}
f_{l}(e_{i},e_{j})=0.
\end{equation*}%
An elementary observation is that $\mathcal{R}_{odd}^{orth}\subset \mathcal{R%
}_{odd}$ [$\mathcal{R}_{odd,sym}^{orth}\subset \mathcal{R}_{odd,sym}$] and
\begin{equation*}
\begin{array}{ccc}
(n,m,k)\in \mathcal{R}_{odd}~\Rightarrow ~(n,m-1,k)\in \mathcal{R}%
_{odd}^{orth} &  & \left[ (n,m,k)\in \mathcal{R}_{odd,sym}~\Rightarrow
~(n,m-1,k)\in \mathcal{R}_{odd,sym}^{orth}\right]%
\end{array}%
\end{equation*}%
by putting inner product on $\mathbb{R}^{n}$ for $f_{m}$.

\subsection*{Generalized Makeev problem}

Let $H=\{x\in \R^{n}~|~\langle x,v\rangle =\alpha \}$ be an affine hyperplane in $\R^{n}$. 
Here $v$ is a vector in $\R^{n}$ and $\alpha \in\R$ some constant. 
The affine hyperplane $H$ determines two open halfspaces%
\begin{equation*}
H^{-}=\{x\in\R^{n}~|~\langle x,v\rangle <\alpha \}\text{ and }H^{+}=\{x\in\R^{n}~|~\langle x,v\rangle >\alpha \}\text{.}
\end{equation*}%
Let $\mathcal{H}=\{H_{1},H_{2},\ldots ,H_{k}\}$ be an arrangement of affine hyperplanes in $\mathbb{R}^{d}$.
An \textit{orthant} of the arrangement $\mathcal{H}$ is an intersection of halfspaces $\mathcal{O}=H_{1}^{\alpha_{1}}\cap\cdots\cap H_{k}^{\alpha _{k}}$, for some $\alpha _{j}\in \Z_2$. 
For convenience we assume that $\Z_2=\left( \{+1,-1\},\cdot\right)$ with obvious abbreviation $H^{+1}\equiv H^{+}$ and $H^{-1}\equiv H^{-}$. 
There are $2^{k}$ orthants determined by $\mathcal{H}$. 
The orthants are not necessary non-empty.  
They can be indexed by elements of the group $\left(\Z_2\right) ^{k}$ in a natural way.

\noindent Let $\mu $ be an absolutely continuous probabilistic measure on $\R^{n}$. 
The arrangement $\mathcal{H}$ \textit{equiparts} the measure $\mu $ if for each orthant $\mathcal{O}$ determined by the arrangement $\mu (\mathcal{O})=%
\tfrac{1}{2^{k}}\mu (\R^{n})$.

\medskip

Generalized Makeev problem is to determine the set $\mathcal{M}\subset\mathbb{N}^{4}$ [$\mathcal{M}^{orth}\subset \mathbb{N}^{4}$] of all quadruples $(n,m,k,l)$, where $1\leq l\leq k$, with the property that for every collection of $m$ absolutely continuous probabilistic measures $\mu _{1},\ldots ,\mu_{m}$ on $\mathbb{R}^{n}$ there exist $k$ [mutually orthogonal] hyperplanes $H_{1},\ldots ,H_{k}$ such that any $l$ of them equipart all the measures.

\noindent It is obvious that $\mathcal{M}^{orth}\subset \mathcal{M}$.
Moreover, by taking $\mu _{m}$ to be the uniform probability measure on the
unit ball in $\mathbb{R}^{n}$ we can derive that%
\begin{equation*}
(n,m,k,l)\in \mathcal{M}~\Rightarrow ~(n,m-1,k,l)\in \mathcal{M}^{orth}.
\end{equation*}%
The generalized Makeev problem for $l=k$ is known as a generalized Gr\"{u}%
nbaum mass partition problem as introduced by Gr\"{u}nbaum in \cite[4.
Remarks (v)]{Grunb} and further studied by Ramos in \cite{ramos1996} and
Mani-Levitska, S.~Vre\'{c}ica, R.~\v{Z}ivaljevi\'{c} in \cite{zvm2006}.

\section{Statement of main results}

Let $A=\mathbb{F}_{2}[t_{1},\ldots ,t_{k}]$ denote the polynomial algebra
with variables $t_{1},\ldots ,t_{k}$ of degree $1$. Then

\begin{equation*}
w_{1}=t_{1}+\dots +t_{k},\ldots ,w_{k}=t_{1}t_{2}\dots t_{k}
\end{equation*}%
are elementary symmetric polynomials in $A$ with the respect to permutation
of variables. Set for $l\geq 1$,
\begin{equation*}
\bar{w}_{l}=\sum_{\substack{ i_{1},i_{2},\ldots ,i_{k}\geq 0  \\ %
i_{1}+2i_{2}+\dots +ki_{k}=l}}\binom{i_{1}+\dots +i_{k}}{i_{1}\ i_{2}\
\ldots \ i_{k}}~w_{1}^{i_{1}}\dots w_{k}^{i_{k}},
\end{equation*}%
where $\binom{i_{1}+\cdots +i_{k}}{i_{1}\ i_{2}\ \ldots \ i_{k}}$ stands for $%
\frac{\left( i_{1}+\cdots +i_{k}\right) !}{\left( i_{1}\right) !~\ldots ~\left(
i_{k}\right) !}$ modulo $2$.

\medskip

\subsection{Rattray type results}

These results give sufficient conditions for a triple $(n,m,k)$ to be in $\mathcal{R}_{\ast }^{\ast }$ and can be formulated in the following way.

\begin{thm}
\label{ram-rattray-st} Let $(n,m,k)\in \mathbb{N}^{3}$. Then
\begin{compactenum}[\rm(a)]
\item $\prod_{1\leq i<j\leq k}(t_{i}+t_{j})^{2m}\notin \langle
t_{1}^{n},\ldots ,t_{k}^{n}\rangle ~\Longrightarrow ~(n,m,k)\in \mathcal{R}%
_{odd},$

\item $\prod_{1\leq i<j\leq k}(t_{i}+t_{j})^{m}\notin \langle t_{1}^{n},\ldots
,t_{k}^{n}\rangle ~\Longrightarrow ~(n,m,k)\in \mathcal{R}_{odd,sym},$

\item $\prod_{1\leq i<j\leq k}(t_{i}+t_{j})^{2m}\notin \langle \bar{w}%
_{n-k+1},\ldots ,\bar{w}_{n}\rangle $~$\Longrightarrow ~(n,m,k)\in \mathcal{R%
}_{odd}^{orth}$,

\item $\prod_{1\leq i<j\leq k}(t_{i}+t_{j})^{m}\notin \langle \bar{w}%
_{n-k+1},\ldots ,\bar{w}_{n}\rangle $~$\Longrightarrow ~(n,m,k)\in \mathcal{R%
}_{odd,sym}^{orth}$.
\end{compactenum}
\end{thm}

\begin{rem}
The degree of the polynomial 
\[
\prod_{1\leq i<j\leq k}(t_{i}+t_{j})=\det\left( t_{i}^{j-1}\right) _{i,j=1}^{k}
\]
is at most $\frac{1}{2}k(k-1)$ and degree of each variable is at most $k-1$. Therefore,
\begin{equation}
\begin{array}{ccccc}
(k-1)m<n & \Longrightarrow & \prod_{1\leq i<j\leq k}(t_{i}+t_{j})^{m}\notin
\langle t_{1}^{n},\ldots ,t_{k}^{n}\rangle & \Longrightarrow & \left(
n,m,k\right) \in \mathcal{R}_{odd,sym}.%
\end{array}
\label{eq:bound-1}
\end{equation}%
Similarly, $2(k-1)m<n$ implies $(n,m,k)\in\mathcal{R}_{odd}$.
\end{rem}

\begin{rem}
Direct application of the criterion (d) of the theorem, for example, implies
that $\left( 3,2,2\right) $, $\left( 4,1,2\right) $, $\left( 4,2,2\right) $,
$\left( 5,m,2\right) $ for $1\leq m\leq 6$ and $\left( 5,1,3\right) $ are
elements of $\mathcal{R}_{odd,sym}^{orth}$. The most striking example is
that $\left( 5,6,2\right) \in \mathcal{R}_{odd,sym}^{orth}$ since the triple
does not fulfill even the inequality bound from the previous remark for
being element of $\mathcal{R}_{odd,sym}$. The fact $\left( 5,6,2\right) \in
\mathcal{R}_{odd,sym}^{orth}$ is the consequence of%
\begin{equation*}
\left( t_{1}+t_{2}\right)
^{6}=t_{1}^{6}+t_{1}^{4}t_{2}^{2}+t_{1}^{2}t_{2}^{4}+t_{2}^{6}\notin \langle
\bar{w}_{4},\bar{w}_{5}\rangle
\end{equation*}%
where%
\begin{equation*}
\begin{array}{lllll}
\bar{w}_{4} & = & w_{1}^{4}+w_{1}^{2}w_{2}+w_{2}^{2} & = &
t_{1}^{4}+t_{1}^{3}t_{2}+t_{1}^{2}t_{2}^{2}+t_{1}t_{2}^{3}+t_{2}^{4}, \\
\bar{w}_{5} & = & w_{1}^{5}+w_{1}w_{2}^{2} & = &
t_{1}^{5}+t_{1}^{4}t_{2}+t_{1}^{3}t_{2}^{2}+t_{1}^{2}t_{2}^{3}+t_{2}^{5},%
\end{array}%
\end{equation*}%
and $w_{1}=t_{1}+t_{2}$, $w_{2}=t_{1}t_{2}$.
\end{rem}

\medskip

Let us present some immediate consequences of Theorem \ref{ram-rattray-st} that generalize results from~\cite{mak2007-2}.

\begin{cor}
\label{ram-rattray-map} Let $\left( n,k,m\right) \in \mathcal{R}%
_{odd,sym}^{orth}$.
\begin{compactenum}[\rm(a)]
\item For every collection $\phi _{1},\ldots ,\phi _{m}$ of $m$ odd maps $%
S^{n-1}\rightarrow S^{n-1}$ there exists an orthonormal $k$-frame $\left(
e_{1},\ldots ,e_{k}\right) \in V_{n}^{k}$ such that for any $1\leq l\leq m$
the set $(\phi _{l}(e_{1}),\ldots ,\phi _{l}(e_{k}))$ is an orthonormal
frame too.

\item For every collection $g_{1},\ldots ,g_{m}$ of $m$ continuous even
functions $\mathbb{R}^{n}\rightarrow \mathbb{R}$ there exists an orthonormal
$k$-frame $\left( e_{1},\ldots ,e_{k}\right) \in V_{n}^{k}$ such that for
any $1\leq l\leq m$ and $1\leq i<j\leq k$
\begin{equation*}
g_{l}(e_{i}+e_{j})=g_{l}(e_{i}-e_{j}).
\end{equation*}
\end{compactenum}
\end{cor}

\begin{proof}
For the first claim take $f_{l}(x,y)=(\phi _{l}(x),\phi _{l}(y))$ and apply Theorem~\ref{ram-rattray-st}, while for the second one take $f_{l}(x,y)=g_{l}(x+y)-g_{l}(x-y)$.
\end{proof}

\medskip

In some particular cases the obvious inequality bound (\ref{eq:bound-1}) can
be substantially improved by more precise cohomology computations.

\begin{thm}
\label{ram-2frames-improved}Let $n\in \mathbb{N}$ and $P(n)=\min \left\{
2^{s}~|~s\in \mathbb{N},2^{s}\geq n\right\} $. Then
\begin{equation*}
P(n)\geq m+2~\Longleftrightarrow ~n\geq \tfrac{1}{2}P(m+2)+1~\Longrightarrow
~(n,m,2)\in \mathcal{R}_{odd,sym}^{orth}.
\end{equation*}
\end{thm}

A further improvement of this result is possible, relating the Rattray problem for $2$-frames to the famous problem of embedding of projective spaces into a Euclidean space.

\begin{thm}
\label{ram-2frames-proj-emb}
If $\mathbb RP^{n-1}$ cannot be embedded into $\mathbb R^m$ because of the ``deleted square obstruction'', then
$$
(n,m,2)\in \mathcal R_{odd,symm}^{orth}.
$$
\end{thm}

\begin{rem}
The \emph{deleted square obstruction} for an embedding $M\to \R^m$ is the obstruction to the existence of a $\Z_2$-equivariant map $\left(M\times M\right)\setminus \Delta(M) \to S^{m-1}$.
Here $\Z_2$ acts on the deleted square $\left(M\times M\right)\setminus\Delta(M)$ by interchanging coordinates and on $S^{m-1}$ antipodally. 
The Haefliger theory~\cite{hae1962} states that in the range $m\ge \frac{3n}{2}$ (the \emph{metastable range}) this is the only obstruction for embedding.
The results in~\cite{dav1984} (see also the table~\cite{dav-table} for some low-dimensional cases) show that asymptotically the required inequality for embedding of the projective space has the form $m\ge 2n - O(\log n)$, i.e., falls into the metastable range. 
It follows that for large enough $n$ the condition $(n,m,2)\in \mathcal R_{odd,symm}^{orth}$ also has the asymptotic form $m\le 2n - O(\log n)$.
\end{rem}

Let us state more results in case $k=3$. 
If we want to calculate in mod $2$ equivariant cohomology, we may consider the Sylow subgroup $W_3^{(2)} = D_8\times \Z_2$ ($D_8$ is the square group). 
We obtain the following algebraic criterion.

\begin{thm}
\label{ram-3frames}
Consider the graded algebra $\mathbb F_2[x,y,w,t]$ with $\dim x = \dim y = \dim t = 1$, $\dim w = 2$, and relation $xy=0$. Put
\begin{compactenum}
\item $w_* = (1 + x + y + w)(1+t)$;
\item $\bar w_* = (w_*)^{-1}$.
\end{compactenum}
In the above notation, if $y^m(t^2 + t(x+y) + w)^m \not\in \langle\bar w_{n-2}, \bar w_{n-1}, \bar w_{n}\rangle$
then $(n, m, 3)\in \mathcal R_{odd,symm}^{orth}$.
\end{thm}

\begin{rem}
It can be checked ``by hand'' than $(3,1,3)\in \mathcal R_{odd,symm}^{orth}$, i.e., the Rattray theorem for $n=3$ follows from this theorem.
\end{rem}

The results of Rattray type can be extended also in the following direction.
It can be asked in addition for the "diagonal" values $f_{l}(e_{i},e_{i})$
to be equal.

\begin{thm}
\label{ram-rattray-unit} Let $k$ and $m$ be positive integers. There exists
a function $n:\mathbb{N}^{2}\rightarrow \mathbb{N}$ such that for every $%
n\geq n(k,m)$ and any collection $f_{1},\ldots ,f_{m}$ of $m$ odd functions $%
S^{n-1}\times S^{n-1}\rightarrow \mathbb{R}$ there exists an orthonormal $k$%
-frame $\left( e_{1},\ldots ,e_{k}\right) \in V_{n}^{k}$ such that for any $%
1\leq l\leq m$ and $1\leq i<j\leq k$
\begin{eqnarray*}
f_{l}(e_{i},e_{j})=0&\text{~~~~~~~~and~~~~~~~~}&f_{l}(e_{1},e_{1})=\ldots=f_{l}(e_{k},e_{k}).
\end{eqnarray*}

\end{thm}

\begin{rem}
Description of the function $n(k,m)$ remains a challenging open problem.
\end{rem}

\smallskip

The final result of Rattray type we present is the following theorem.

\begin{thm}
\label{rattray-gen}Let $\psi :S^{n-1}\rightarrow S^{m-1}$ be an odd
continuous map and $1\leq k\leq n$. For any linear subspace $L\subseteq
\mathbb{R}^{m}$ of codimension $n-k$ there exists an orthonormal $k$-frame $%
(e_{1},\ldots ,e_{k})$ in $\mathbb{R}^{n}$ such that $(\psi (e_{1}),\ldots
,\psi (e_{k}))$ is an orthonormal $k$-frame in $L$.
\end{thm}

\begin{rem}
This theorem implies that $m$ must be at least $n$ (when considered
$k=n$), i.e. it implies the Borsuk--Ulam theorem.
\end{rem}

\subsection{Makeev type results}

The following theorem gives sufficient conditions for $(n,m,k,l)$ to be in $\mathcal{M}^{\ast }$.

\begin{thm}
\label{Makeev-general} Let $(n,m,k,l)\in \mathbb{N}^{4}$. Then
\begin{compactenum}[\rm(a)]

\item $\prod_{\substack{ s_{1},\ldots ,s_{k}\in \mathbb{Z}_{2} \\ 1\leq
s_{1}+\cdots +s_{k}\leq l}}(s_{1}t_{1}+s_{2}t_{2}+\cdots
+s_{k}t_{k})^{m}\notin \langle t_{1}^{n+1},\ldots ,t_{k}^{n+1}\rangle
~\Longrightarrow ~(n,m,k,l)\in \mathcal{M},$

\item $\frac{1}{t_{1}\cdots t_{k}}\prod_{\substack{ s_{1},\ldots ,s_{k}\in \mathbb{%
Z}_{2} \\ 1\leq s_{1}+\cdots +s_{k}\leq l}}(s_{1}t_{1}+s_{2}t_{2}+\cdots
+s_{k}t_{k})^{m}\notin \langle \bar{w}_{n-k+1},\ldots ,\bar{w}_{n}\rangle $~$%
\Longrightarrow ~(n,m,k,l)\in \mathcal{M}^{orth}$.

\end{compactenum}
\end{thm}

\begin{rem}
By considering maximal degree of the test polynomial in every variable we
can get a rough bound
\begin{equation*}
n\geq m\left( \sum_{i=0}^{l}\binom{k-1}{i}\right) ~\Longrightarrow
~(n,m,k,l)\in \mathcal{M}\text{.}
\end{equation*}
\end{rem}

\begin{rem}
Notice that for $m=1$ and $l=2$ algebraic conditions of Theorem \ref%
{Makeev-general} (b) and Theorem \ref{ram-rattray-st} (d) coincide.
\end{rem}

\begin{rem}
For $l=k$, the case (a) is equivalent to the main result of the paper by
Mani-Levitska, S.~Vre\'{c}ica, R.~\v{Z}ivaljevi\'{c} \cite[Theorem~39]{zvm2006}. They
obtained that%
\begin{equation*}
n\geq 2^{q+k-1}+r~\Longrightarrow ~(n,2^{q}+r,k,k)\in \mathcal{M}
\end{equation*}%
where $m=2^{q}+r$ and $0\leq r\leq 2^{q}-1$.
\end{rem}

Similar to Theorem~\ref{ram-2frames-proj-emb}, we prove another particular result on partitioning measures by pairs of hyperplanes. This result is a projective analogue of the ``ham sandwich'' theorem~\cite{st1942,ste1945}, the concept of ``projective measure partitions'' is due to Benjamin Matschke (private communication).

\begin{thm}
\label{proj-ham-sandwich}
Suppose $\mathbb RP^{n-1}$ cannot be embedded into $\mathbb R^m$ because of the ``deleted square obstruction''. Let $\mu_0,\ldots, \mu_m$ be $m+1$ absolutely continuous probabilistic measures on $\mathbb RP^{n-1}$. Then there exists a pair of hyperplanes $H_1, H_2\subseteq \mathbb RP^{n-1}$, partitioning every measure $\mu_i$ into two equal parts.
\end{thm}

\begin{rem}
A single hyperplane does not partition a projective space, but two hyperplanes partition it into two parts. 
\end{rem} 

\begin{rem}
The condition is asymptotically $m\le 2n - O(\log n)$, as in Theorem~\ref{ram-2frames-proj-emb}.
\end{rem}

\section{Equivariant cohomology of the Stiefel manifold}

Let $V_{n}^{k}$ denote the Stiefel manifold of all orthonormal $k$-frames
in $\mathbb{R}^{n}$. Any subgroup $G\subseteq \mathrm{O}(k)$ acts naturally
on $k$-frames by
\begin{equation*}
(e_{1},\ldots ,e_{k})\cdot g=\left(
\sum_{j}e_{j}s_{j1},\ldots,\sum_{j}e_{j}s_{jk}\right)
\end{equation*}%
where $(e_{1},\ldots ,e_{k})\in V_{n}^{k}$ and $g=\left( s_{ij}\right)_{i,j=1}^{k}\in \mathrm{O}(k)$. 
The action is right, but it transforms in a left action in the usual way $g\cdot (e_{1},\ldots ,e_{k}):=(e_{1},\ldots
,e_{k})\cdot g^{-1}$.

\medskip

In this section we compute the Fadell--Husseini index of the Stiefel manifold $V_{n}^{k}$ with the respect to the action of any subgroup $G\subseteq \mathrm{O}(k)$ and coefficients $\mathbb{F}_{2}$, i.e., we determine the generators of the following ideal
\begin{equation*}
\mathrm{Index}_{G,\mathbb{F}_{2}}V_{n}^{k}=\ker \left( H^{\ast }(G;\mathbb{F}%
_{2})\longrightarrow H^{\ast }(\mathrm{E}G\times _{G}V_{n}^{k};\mathbb{F}%
_{2})\right) .
\end{equation*}%
In particular, we determine explicitly the index with respect to the subgroup $\mathbb{Z}_{2}^{k}$ of diagonal matrices with $\{-1,1\}$ entries on diagonal. One
description of the index $\mathrm{Index}_{\mathbb{Z}_{2}^{k},\mathbb{F}%
_{2}}V_{n}^{k}$ is given in the initial paper of Fadell and Husseini \cite[%
Theorem 3.16, page 78]{Fadell--Husseini}.

\subsection{ ~}

The cohomology of the Stiefel manifold $V_{n}^{k}$ with $\mathbb{F}_{2}$ coefficients is the quotient algebra (consult~\cite{bor1953})
\begin{equation*}
H^{\ast }\left( V_{n}^{k};\mathbb{F}_{2}\right) =\mathbb{F}%
_{2}[e_{n-k},\ldots,e_{n-1}]/\mathcal{J}_{n}^{k}
\end{equation*}%
where $\deg e_{i}=i$ and $\mathcal{J}_{n}^{k}$ is the ideal generated by the
relations%
\begin{equation*}
\begin{array}{lll}
e_{i}^{2}=e_{2i} &  & \text{for }2i\leq n-1 \\
e_{i}^{2}=0 &  & \text{for }2i\geq n.%
\end{array}%
\end{equation*}

\medskip

In what follows, for a vector bundle $F\rightarrow \xi \rightarrow B$ we denote by $w_{i}(\xi )\in H^{i}\left( B;\mathbb{F}_{2}\right) $ the associated Stiefel--Whitney classes, by $\bar{w}_{i}(\xi )\in
H^{i}\left( B;\mathbb{F}_{2}\right) $ its dual Stiefel--Whitney classes, $i\geq 0$. There is a relation between these classes expressed via the total class by $w\cdot \bar{w}=1$ or particularly for $l\geq 1$ by
\begin{equation*}
\bar{w}_{l}(\xi )=\sum_{\substack{ i_{1},i_{2},\ldots ,i_{k}\geq 0  \\ %
i_{1}+2i_{2}+\dots +ki_{k}=l}}\binom{i_{1}+\dots +i_{k}}{i_{1}\ i_{2}\
\ldots \ i_{k}}~w_{1}^{i_{1}}(\xi )\dots w_{k}^{i_{k}}(\xi ).
\end{equation*}

\medskip

Let us recall that:
\begin{compactenum}[\rm(a)]
\item the Grassmann manifold $G^{k}(\mathbb{R}^{\infty })$ of all $k$-flats in $\mathbb{R}^{\infty }$ is
the classifying space of the group $\mathrm{O}(k)$ and we denote $G^{k}(\mathbb{R}^{\infty })$
also by $\mathrm{BO}(k)$,

\item the Stiefel manifold $V_\infty^{k}$ of all $k$-frames in $\mathbb{R}^{\infty }$ as
a contractible free $\mathrm{O}(k)$ space serves as a model for $\mathrm{EO}(k)
$,

\item the associated canonical bundle:
\begin{equation*}
\begin{array}{ccccc}
\mathbb{R}^{k} & \longrightarrow  & \gamma ^{k} & \longrightarrow  & G^{k}(\mathbb{R}^{\infty })%
\end{array}
\end{equation*}%
can be seen as a Borel construction of the $\mathrm{O}(k)$-space $\mathbb{R}%
^{k}$ [where the action is given by the matrix multiplication from the left]:%
\begin{equation*}
\begin{array}{ccccc}
\mathbb{R}^{k} & \longrightarrow  & \mathrm{EO}(k)\times _{\mathrm{O}(k)}%
\mathbb{R}^{k} & \longrightarrow  & \mathrm{BO}(k)
\end{array}%
\end{equation*}

\item the cohomology of the Grassmanian $G^{k}\left(\mathbb{R}^{\infty}\right)\approx \mathrm{BO}(k)$ with
coefficients in $\mathbb{F}_{2}$ is the polynomial algebra generated by the Stiefel--Whithey classes $w_{1},\ldots ,w_{k}$ of the canonical vector bundle $\gamma^k$:
\begin{equation*}
H^{\ast }\left( \mathrm{BO}(k);\mathbb{F}_{2}\right) =\mathbb{F}_{2}\left[
w_{1},\ldots,w_{k}\right].
\end{equation*}
\end{compactenum}

\smallskip

\noindent Now we state a very useful result from~\cite{bor1953} (see also~\cite[Theorem 3.3]{Jaw}).

\begin{proposition}
\label{Prop:Jaw}Let $(E_{i}^{\ast ,\ast },d_{i})_{i\geq 2}$ denote the
Leray--Serre spectral sequence associated with the Borel construction
\begin{equation*}
\begin{array}{ccccc}
\mathbb{R}^{k} & \longrightarrow & \mathrm{EO}(k)\times _{\mathrm{O}(k)}%
\mathbb{R}^{k} & \longrightarrow & \mathrm{BO}(k)%
\end{array}%
.
\end{equation*}%
Then
\begin{equation*}
\mathrm{Index}_{\mathrm{O}(k),\mathbb{F}_{2}}V_{n}^{k}=\langle \bar{w}%
_{n-k+1},\ldots,\bar{w}_{n}\rangle \subset \mathbb{F}_{2}\left[ w_{1},\ldots ,w_{k}%
\right]
\end{equation*}
where $\bar{w}_{i}=\bar{w}_{i}\left( \gamma ^{k}\right) =d_{i-1}\left(e_{i-1}\right)$.
\end{proposition}

\subsection{ ~}

The Borel construction is a functorial construction and therefore there is a morphism of fiber bundles induced by the inclusion $\iota \colon G\subseteq\mathrm{O}(k)$:

\begin{diagram}
\mathrm{EO}(k)\times _{G}
V_{n}^{k}&\rTo&\mathrm{EO}(k)\times_{\mathrm{O}(k)}V_{n}^{k}\\
\dTo{\pi}& &\dTo{\mu}&  \\
\mathrm{B}G&\rTo^{B\iota} & \mathrm{BO}(k) \\
\end{diagram}

\noindent In the bundle on the left, $\mathrm{EO}(k)$ is used as a model for $\mathrm{E}G$. The action of $\mathrm{O}(k)$ on the Stiefel manifold $V_{n}^{k}$ is free. Therefore, the $E_{\infty }^{p,q}$-term of the Leray--Serre spectral sequence for the fibration $\mathrm{EO}(k)\times _{\mathrm{O}(k)}V_{n}^{k}\rightarrow \mathrm{BO}(k)$ has to vanish for $p+q>\mathrm{dim}~V_{n}^{k}$. Furthermore, $\mathrm{O}(k)$ acts trivially on the cohomology $H^{\ast }(V_{n}^{k};\mathbb{F}_{2})$ and so by Proposition~\ref{Prop:Jaw} we have that $d_{i}\left( e_{i}\right) =\bar{w}_{i+1}$ for $n-k\leq i\leq n-1$. Here $d_{i}$ denotes the $i$-th differential of the Leray--Serre spectral sequence. The morphism of the bundles we considered induces a morphism of the associated Leray--Serre spectral sequences as well. The morphism in the $E_{2}$-term on the $0$-column is the identity and on the $0$-row determines the restriction morphism $\iota ^{\ast }=$\textrm{%
res}$_{G}^{\mathrm{O}(k)}$. Thus,
\begin{eqnarray*}
\mathrm{Index}_{G,\mathbb{F}_{2}}V_{n}^{k} &=&\ker \pi ^{\ast }=\mathrm{res}%
_{G}^{\mathrm{O}(k)}\left( \ker \mu ^{\ast }\right) =\mathrm{res}_{G}^{%
\mathrm{O}(k)}\left( \langle \bar{w}_{n-k+1},\ldots ,\bar{w}_{n}\rangle \right)
\\
&=&\langle \mathrm{res}_{G}^{\mathrm{O}(k)}\left( \bar{w}_{n-k+1}\right)
,\ldots ,\mathrm{res}_{G}^{\mathrm{O}(k)}\left( \bar{w}_{n}\right) \rangle .
\end{eqnarray*}%
We have proved the following claim.

\begin{proposition}
$\mathrm{Index}_{G,\mathbb{F}_{2}}V_{n}^{k}=\langle \mathrm{res}_{G}^{%
\mathrm{O}(k)}\left( \bar{w}_{n-k+1}\right) ,\ldots ,\mathrm{res}_{G}^{\mathrm{O}%
(k)}\left( \bar{w}_{n}\right) \rangle .$
\end{proposition}

\subsection{ ~}

\noindent In the final step we identify the restriction morphism $\mathrm{res}_{G}^{\mathrm{O}(k)}$. Consider $\mathbb{R}^{k}$ as an $\mathrm{O}(k)$-space where the action is given by the left matrix multiplication. The inclusion $\iota \colon G\subseteq \mathrm{O}(k)$ gives to $\mathbb{R}^{k}$ the structure of $G$-space. Again, there is a morphism of associated Borel constructions, which in this case is also a morphism of vector bundles:

\begin{diagram}
\mathrm{EO}(k)\times _{G}
\mathbb{R}^{k}&\rTo&\mathrm{EO}(k)\times_{\mathrm{O}(k)}\mathbb{R}^{k}\\
\dTo{\phi}& &\dTo{\psi}&  \\
\mathrm{B}G&\rTo^{B\iota} & \mathrm{BO}(k) \\
\end{diagram}

\noindent The naturality of the Stiefel--Whitney classes implies that
\begin{equation*}
w_{i}(\mathrm{EO}(k)\times _{G}\mathbb{R}^{k})=\iota ^{\ast }(w_{i})=\mathrm{%
res}_{G}^{\mathrm{O}(k)}(w_{i})
\end{equation*}%
and consequently
\begin{equation*}
\bar{w}_{i}(\mathrm{EO}(k)\times _{G}\mathbb{R}^{k})=\mathrm{res}_{G}^{%
\mathrm{O}(k)}(\bar{w}_{i}).
\end{equation*}%
Thus we have proved the following fact.

\begin{proposition}
\label{prop:StiefelG}$\mathrm{Index}_{G,\mathbb{F}_{2}}V_{n}^{k}=\langle
\bar{w}_{n-k+1}(\mathrm{EO}(k)\times _{G}\mathbb{R}^{k}),\ldots ,\bar{w}_{n}(%
\mathrm{EO}(k)\times _{G}\mathbb{R}^{k})\rangle .$
\end{proposition}

\subsection{ ~}

Let $G=\mathbb{Z}_{2}^{k}$ be the subgroup of diagonal matrices with $\{-1,1\}$ entries. 
Let $A:=H^{\ast }\left( \mathbb{Z}_{2}^{k};\mathbb{F}_{2}\right) =\mathbb{F}_{2}[t_{1},\ldots ,t_{k}]$ be the polynomial
algebra with variables $t_{1},\ldots ,t_{k}$ of degree $1$.

It is well known that the $k$-dimensional real $\mathbb{Z}_{2}^{k}$-representation $\mathbb{R}^{k}$ can be decomposed into the sum of $1$-dimensional irreducible real $\mathbb{Z}_{2}^{k}$-representation. The total
Stiefel--Whitey class of $\mathrm{EO}(k)\times _{\mathbb{Z}_{2}^{k}}\mathbb{R}^{k}$ is given by
\begin{equation*}
w\left( \mathrm{EO}(k)\times _{\mathbb{Z}_{2}^{k}}\mathbb{R}^{k}\right)
=\prod\limits_{i=1}^{k}\left( 1+t_{i}\right) =1+\omega _{1}+\cdots +\omega _{k}
\end{equation*}%
where $\omega _{i}$ denotes both: the elementary symmetric polynomial of degree $i$ in variables $t_{1},\ldots ,t_{k}$ and the $i$-th Stiefel--Whitney class of $w_{i}\left( \mathrm{EO}(k)\times _{\mathbb{Z}_{2}^{k}}\mathbb{R}^{k}\right) $. For example, $\omega _{1}=t_{1}+t_{2}+\cdots+t_{k}$ while $\omega _{k}=t_{1}t_{2}\cdots t_{k}$. Finally, we obtain the following result.

\begin{proposition}
\label{Prop:IndexOfStiefel}Let $\bar{\omega}_{l}=\sum_{\substack{ %
i_{1},i_{2},\ldots ,i_{k}\geq 0  \\ i_{1}+2i_{2}+\dots +ki_{k}=l}}\binom{%
i_{1}+\dots +i_{k}}{i_{1}\ i_{2}\ \ldots \ i_{k}}~\omega _{1}^{i_{1}}\dots
\omega _{k}^{i_{k}}$, for $l\geq 1$, and then%
\begin{equation*}
\mathrm{Index}_{\mathbb{Z}_{2}^{k},\mathbb{F}_{2}}V_{n}^{k}=\langle \bar{%
\omega}_{n-k+1},\ldots ,\bar{\omega}_{n}\rangle \subset A.
\end{equation*}
\end{proposition}

\section{Proof of Rattray type results}

\label{Section:Proof of Rattray}

\subsection{ ~}

The proofs of these results will be done via the configuration space~/~test
map method. There are two different natural configuration spaces of interest:%
\begin{equation*}
\begin{array}{lllll}
X & = & \left( S^{n-1}\right) ^{k} & = & \text{the space of all collections
of }k\text{ vectors on the sphere }S^{n-1}\text{,} \\
Y & = & V_{n}^{k} & = & \text{the space of all orthogonal }k\text{-frames in
}\mathbb{R}^{n}%
\end{array}%
.
\end{equation*}%
The group $W_{k}=(\mathbb{Z}_{2})^{k}\rtimes \Sigma _{k}\subset \mathrm{O}%
(k) $ acts naturally on both configurations spaces. For the generators $%
\varepsilon _{1},\ldots ,\varepsilon _{n}$ of the component $(\Z_{2})^{n}$ and $%
\left( e_{1},\ldots ,e_{k}\right) \in X$ or $Y$ the action is given by%
\begin{equation*}
\varepsilon _{i}\cdot \left( e_{1},\ldots ,e_{k}\right) =\left(
e_{1}^{\prime },\ldots ,e_{k}^{\prime }\right) \text{ where }e_{i}^{\prime
}=-e_{i}\text{ and }e_{j}^{\prime }=e_{j}\text{ for }j\neq i ,
\end{equation*}%
and for the permutation $\pi \in \Sigma _{k}$ by%
\begin{equation*}
\pi \cdot \left( e_{1},\ldots ,e_{k}\right) =\left( e_{\pi (1)},\ldots
,e_{\pi (k)}\right) .
\end{equation*}

\medskip

Let us consider the space $M_{k}$ of all real $k\times k$-matrices as a real
$\mathrm{O}(k)$-representation with respect to the action
\begin{equation*}
m\mapsto gmg^{-1}
\end{equation*}%
where $m\in M_{k}$ and $g$ is $k\times k$-matrix representing an element of $%
\mathrm{O}(k)$. Then $M_{k}$ has a structure of a real $W_{k}$%
-representation via the inclusion map $W_{k}\hookrightarrow \mathrm{O}(k)$.
Consider following real vector subspaces of $M_{k}$:%
\begin{equation}
\begin{tabular}{l}
$R_{k}$ of all $k\times k$ symmetric matrices with zeros on the diagonal, \\
$U_{k}$ of all $k\times k$ matrices with zeros on the diagonal, and \\
$I_{k}$ of all $k\times k$ matrices with zeros outside the diagonal and
trace zero.%
\end{tabular}
\label{eq:DefR_kU_k}
\end{equation}
These are all real $W_{k}$-subrepresentations of $M_{k}$. 
Moreover, when we consider only the subgroup $(\Z_2)^k$ there is a decomposition $U_{k}\cong R_{k}\oplus R_{k}$ of $(\Z_2)^k$-representation.

\smallskip

For an odd [and symmetric] function $f:S^{n-1}\times S^{n-1}\rightarrow
\mathbb{R}$ and $k$-vectors [$k$-frame] $\left( e_{1},\ldots ,e_{k}\right)$, we denote by:

\begin{compactitem}
\item $\mu _{f}\left( e_{1},\ldots ,e_{k}\right) \in U_{k}$ [$\mu _{f}\left(
e_{1},\ldots ,e_{k}\right) \in R_{k}$] the matrix given by entries%
\begin{equation*}
\left( \mu _{f}\left( e_{1},\ldots ,e_{k}\right) \right) _{ij}=\left\{
\begin{array}{ll}
f\left( e_{i},e_{j}\right)  & ,~i\neq j \\
0 & ,~i=j,%
\end{array}%
\right. 
\end{equation*}

\item $\eta _{f}\left( e_{1},\ldots ,e_{k}\right) \in I_{k}$ the matrix
given by entries%
\begin{equation*}
\left( \eta _{f}\left( e_{1},\ldots ,e_{k}\right) \right) _{ij}=\left\{
\begin{array}{ll}
f\left( e_{i},e_{i}\right) -c & ,~i=j \\
0 & ,~i\neq j,
\end{array}%
\right.
\end{equation*}%
where $c=\frac{1}{k}\left(f\left( e_{1},e_{1}\right) +\cdots +f\left( e_{k},e_{k}\right) \right)$.
\end{compactitem}

\subsection{\label{ram-rattray-proof}Proof of Theorem~\protect\ref%
{ram-rattray-st}}

Let $(n,m,k)\in \mathbb{N}^{3}$ and $f_{1},\ldots ,f_{m}$ be a collection of
$m$ odd [and symmetric] functions $S^{n-1}\times S^{n-1}\rightarrow \mathbb{R%
}$. Let us introduce the test maps for the Rattray problems:%
\begin{equation*}
\begin{array}{cccc}
\tau _{odd}:X\rightarrow U_{k}^{\oplus m}, & \tau _{odd,sym}:X\rightarrow
R_{k}^{\oplus m}, & \tau _{odd}^{orth}:Y\rightarrow U_{k}^{\oplus m}, & \tau
_{odd,sym}^{orth}:Y\rightarrow R_{k}^{\oplus m}.%
\end{array}%
\end{equation*}%
All four test maps are defined by the same formula%
\begin{equation*}
\begin{array}{ccc}
\left( e_{1},\ldots ,e_{k}\right) & \overset{\tau _{\ast }^{\ast }}{%
\longmapsto } & \left( \mu _{f_{r}}\left( e_{1},\ldots ,e_{k}\right) \right)
_{r=1}^{m}%
\end{array}%
\end{equation*}%
assuming appropriate domains and codomains. Have in mind that the test maps
are functions of the collection $f_{1},\ldots ,f_{m}$, even we abbreviate
this from notation. The test maps are all $W_{k}$-equivarian maps and
moreover have the following obvious but very important properties: If for
every collection $f_{1},\ldots ,f_{m}$ of $m$ odd [and symmetric] functions $%
S^{n-1}\times S^{n-1}\rightarrow \mathbb{R}$

\begin{compactitem}
\item $\{\mathbf{0\in }U_{k}^{\oplus m}\}\in \tau _{odd}\left( X\right)$, then
$(n,m,k)\in \mathcal{R}_{odd}$,

\item $\{\mathbf{0\in }U_{k}^{\oplus m}\}\in \tau _{odd}\left( X\right)$, then
$(n,m,k)\in \mathcal{R}_{odd}$,

\item $\{\mathbf{0\in }R_{k}^{\oplus m}\}\in \tau _{odd,sym}\left( X\right)$, then
$(n,m,k)\in \mathcal{R}_{odd,sym}$,

\item $\{\mathbf{0\in }U_{k}^{\oplus m}\}\in \tau _{odd}^{orth}\left( Y\right)$, then
$(n,m,k)\in \mathcal{R}_{odd}^{orth}$,

\item $\{\mathbf{0\in }R_{k}^{\oplus m}\}\in \tau _{odd,sym}^{orth}\left(
Y\right)$, then $(n,m,k)\in \mathcal{R}_{odd,sym}^{orth}$.
\end{compactitem}

\medskip

Let us assume that Theorem \ref{ram-rattray-st} fails in each case.
This means that for a specific collection $%
f_{1},\ldots ,f_{m}$ of $m$ odd [and symmetric] functions $\mathbf{0}\in
U_{k}^{\oplus m}$ or $\mathbf{0}\in R_{k}^{\oplus m}$ is not in the image of
any of the test maps. 
Therefore, we have constructed the following $W_{k}$-equivariant maps
\begin{equation}
\begin{array}{cccc}
X\rightarrow U_{k}^{\oplus m}\backslash \left\{ \mathbf{0}\right\} , &
X\rightarrow R_{k}^{\oplus m}\backslash \left\{ \mathbf{0}\right\} , &
Y\rightarrow U_{k}^{\oplus m}\backslash \left\{ \mathbf{0}\right\} , &
Y\rightarrow R_{k}^{\oplus m}\backslash \left\{ \mathbf{0}\right\} ,%
\end{array}
\label{eq:maps-1}
\end{equation}%
i.e., after $W_{k}$-equivariant homotopy, the $W_{k}$-equivariant maps
\begin{equation}
\begin{array}{cccc}
X\rightarrow S\left( U_{k}^{\oplus m}\right) , & X\rightarrow S\left(
R_{k}^{\oplus m}\right) , & Y\rightarrow S\left( U_{k}^{\oplus m}\right) , &
Y\rightarrow S\left( R_{k}^{\oplus m}\right) .%
\end{array}
\label{eq:maps-2}
\end{equation}%
Obviously all these maps are $\mathbb{Z}_{2}^{k}$-equivariant maps, where $\mathbb{Z}_{2}^{k}$ is the diagonal subgroup of $W_{k}$.

The basic monotonicity property of the Fadell--Husseini index theory \cite%
{Fadell--Husseini} states that when there is a $G$ map $A\rightarrow B$
between $G$-spaces $A$ and $B$ there has to be an inclusion of associated
indexes $\mathrm{Index}_{G,\mathbb{\ast }}A\supseteq \mathrm{Index}_{G,%
\mathbb{\ast }}B$. Using the subgroup $\mathbb{Z}_{2}^{k}$ of $W_{k}$ the
maps (\ref{eq:maps-2}) induce following inclusions%
\begin{equation}
\begin{array}{ccc}
\mathrm{Index}_{\mathbb{Z}_{2}^{k},\mathbb{F}_{2}}~X\supseteq \mathrm{Index}%
_{\mathbb{Z}_{2}^{k},\mathbb{F}_{2}}S\left( U_{k}^{\oplus m}\right) , &  &
\mathrm{Index}_{\mathbb{Z}_{2}^{k},\mathbb{F}_{2}}~X\supseteq \mathrm{Index}%
_{\mathbb{Z}_{2}^{k},\mathbb{F}_{2}}~S\left( R_{k}^{\oplus m}\right) , \\
\mathrm{Index}_{\mathbb{Z}_{2}^{k},\mathbb{F}_{2}}~Y\supseteq \mathrm{Index}%
_{\mathbb{Z}_{2}^{k},\mathbb{F}_{2}}~S\left( U_{k}^{\oplus m}\right) , &  &
\mathrm{Index}_{\mathbb{Z}_{2}^{k},\mathbb{F}_{2}}~Y\supseteq \mathrm{Index}%
_{\mathbb{Z}_{2}^{k},\mathbb{F}_{2}}~S\left( R_{k}^{\oplus m}\right) .%
\end{array}
\label{eq:inclusions}
\end{equation}%
We determine all Fadell--Husseini indexes appearing in (\ref{eq:inclusions}).

\begin{claim*}
With notation already introduced:
\begin{compactenum}[\rm(a)]
\item $\mathrm{Index}_{\mathbb{Z}_{2}^{k},\mathbb{F}_{2}}~X=\langle
t_{1}^{n},\ldots ,t_{k}^{n}\rangle ,$

\item $\mathrm{Index}_{\mathbb{Z}_{2}^{k},\mathbb{F}_{2}}~Y=\langle \bar{\omega%
}_{n-k+1},\ldots ,\bar{\omega}_{n}\rangle ,$

\item $\mathrm{Index}_{\mathbb{Z}_{2}^{k},\mathbb{F}_{2}}~S\left(
R_{k}^{\oplus m}\right) =\langle\prod_{1\leq a<b\leq k}(t_{a}+t_{b})^{m}\rangle,$

\item $\mathrm{Index}_{\mathbb{Z}_{2}^{k},\mathbb{F}_{2}}~S\left(
U_{k}^{\oplus m}\right) =\langle\prod_{1\leq a<b\leq k}(t_{a}+t_{b})^{2m}\rangle.$
\end{compactenum}
\end{claim*}

\begin{proof}
(a) Since the $\mathbb{Z}_{2}^{k}$-action on $X$ is component-wise antipodal the index of $X$ is computed in the paper of Fadell and Husseini~\cite[Example~3.3, p.~76]{Fadell--Husseini}.

\noindent (b) This fact is established in Proposition~\ref{Prop:IndexOfStiefel}.

\noindent (c) Let us denote by $R_{ab}$, for $1\leq a<b\leq k$, the $1$%
-dimension real vector subspace of $R_{k}$ described by%
\begin{equation*}
R_{ab}=\left\{ m\in R_{k}~|~m_{ij}=0\text{ for }(i,j)\notin \{(a,b),(b,a)\}%
\text{ and}~m_{ab}=m_{ba}\in \mathbb{R}\right\} .
\end{equation*}%
The subspace $R_{ab}$ is $\mathbb{Z}_{2}^{k}$-invariant and
\begin{equation*}
\varepsilon _{i}\cdot m=\left\{
\begin{array}{r}
-m \\
m%
\end{array}%
\begin{array}{l}
\text{, for }i\in \{a,b\} \\
\text{, for }i\in \{1,\ldots ,k\}\backslash \{a,b\}%
\end{array}%
\right. .
\end{equation*}%
Moreover, $R_{k}\cong \bigoplus\limits_{1\leq a<b\leq k}R_{ab}$ as a $%
\mathbb{Z}_{2}^{k}$-module. Since the Fadell--Husseini index of a sphere in
this case is a principal ideal generated by the Euler class [= the top
Stiefel--Whitney class] of the vector bundle%
\begin{equation*}
R_{k}\longrightarrow \mathrm{E}\mathbb{Z}_{2}^{k}\times _{\mathbb{Z}%
_{2}^{k}}R_{k}\longrightarrow \mathrm{B}\mathbb{Z}_{2}^{k}
\end{equation*}%
and%
\begin{equation*}
\mathfrak{e}(\mathrm{E}\mathbb{Z}_{2}^{k}\times _{\mathbb{Z}%
_{2}^{k}}R_{k})=\prod\limits_{1\leq a<b\leq k}\mathfrak{e}(\mathrm{E}\mathbb{%
Z}_{2}^{k}\times _{\mathbb{Z}_{2}^{k}}R_{ab})=\prod\limits_{1\leq a<b\leq
k}\left( t_{a}+t_{b}\right) .
\end{equation*}%
For details consult \cite[Proof of Proposition 3.11]{B-Z}. It follows
directly that
\begin{equation*}
\mathfrak{e}(\mathrm{E}\mathbb{Z}_{2}^{k}\times _{\mathbb{Z}%
_{2}^{k}}R_{k}^{\oplus m})=\prod\limits_{1\leq a<b\leq k}\left(
t_{a}+t_{b}\right) ^{m}
\end{equation*}%
and consequently $\mathrm{Index}_{\mathbb{Z}_{2}^{k},\mathbb{F}_{2}}~S\left(
R_{k}^{\oplus m}\right) =\langle \prod\limits_{1\leq a<b\leq k}\left(
t_{a}+t_{b}\right) ^{m}\rangle $.

\noindent (d) Follows from the decomposition $U_{k}\cong R_{k}\oplus R_{k}$
of $\mathbb{Z}_{2}^{k}$-module.
\end{proof}

\medskip

Now, the inclusions (\ref{eq:inclusions}) with just determined indexes imply
that:%
\begin{equation*}
\begin{array}{lll}
\prod\limits_{1\leq a<b\leq k}\left( t_{a}+t_{b}\right) ^{m}\in \langle
t_{1},\ldots,t_{k}\rangle , &  & \prod\limits_{1\leq a<b\leq k}\left(
t_{a}+t_{b}\right) ^{m}\in \langle t_{1},\ldots,t_{k}\rangle , \\
\prod\limits_{1\leq a<b\leq k}\left( t_{a}+t_{b}\right) ^{m}\in \langle \bar{%
\omega}_{n-k+1},\ldots,\bar{\omega}_{n}\rangle , &  & \prod\limits_{1\leq
a<b\leq k}\left( t_{a}+t_{b}\right) ^{m}\in \langle \bar{\omega}_{n-k+1},\ldots,%
\bar{\omega}_{n}\rangle .%
\end{array}%
\end{equation*}%
This gives a \textbf{contradiction} with the assumptions of Theorem \ref%
{ram-rattray-st}. Therefore, all claims of Theorem \ref{ram-rattray-st} hold.

\subsection{\label{k2-sec}Proof of Theorem~\protect\ref{ram-2frames-improved}}

Before starting the proof let us once more isolate an important property of Stiefel--Whitney classes already used in the proof of Theorem~\ref{ram-rattray-st}. Let $H$ be a subgroup of a group $G$ and $V$ a real $G$-representation. Then the following equality between the total Stiefel--Whitney classes holds:
\begin{equation*}
w\left( \mathrm{E}H\times _{H}V\right) =\mathrm{res}_{H}^{G}~\left( w\left(
\mathrm{E}G\times _{G}V\right) \right) ~\Longleftrightarrow ~w_{i}\left(
\mathrm{E}H\times _{H}V\right) =\mathrm{res}_{H}^{G}~\left( w_{i}\left(
\mathrm{E}G\times _{G}V\right) \right) \text{ for all }i\geq 1
\end{equation*}
where $V$ inherits the $H$-representation structure from the inclusion map $H\hookrightarrow G$.

\medskip

In the proof we use the complete group of symmetries $W_{2}=(\mathbb{Z}%
_{2})^{2}\rtimes \mathbb{Z}_{2}=\left( \langle \varepsilon _{1}\rangle
\times \langle \varepsilon _{2}\rangle \right) \rtimes \langle \sigma
\rangle $ which is isomorphic to the dihedral group $D_{8}$. The cohomology
of the dihedral group $D_{8}$ with $\mathbb{F}_{2}$ coefficients is given by%
\begin{equation*}
H^{\ast }(D_{8};\mathbb{F}_{2})=\mathbb{F}_{2}[x,y,w]/\langle xy\rangle .
\end{equation*}%
where $\deg x=\deg y=1$ and $\deg w=2$. Consult \cite[Section IV.1, page 116]%
{admi2004} or \cite[Section 4.2]{B-Z}. In what follows we use the notations
introduced in the paper \cite[Section 4.3.2]{B-Z}. For example subgroup $(%
\mathbb{Z}_{2})^{2}$ is denoted by $H_{1}$, while subgroup $\langle \sigma
\rangle $ is either $K_{4}$ or $K_{5}$. Let us assume for clarity that $%
K_{5}=\langle \sigma \rangle $.

\medskip

Let us consider $W_{2}=D_{8}$ and its already introduced representations $%
R_{2}$ and $\mathbb{R}^{2}$. Computation of the total Stiefel--Whitney class
$w\left( \mathrm{E}(\mathbb{Z}_{2})^{2}\times _{(\mathbb{Z}%
_{2})^{2}}R_{2}\right) $ conducted in Section \ref{ram-rattray-proof}, when
translated into the notation of \cite[Section 4.3.2]{B-Z}, gives us that%
\begin{equation*}
w\left( \mathrm{E}H_{1}\times _{H_{1}}R_{2}\right) =1+\left( a+a+b\right)
=1+b
\end{equation*}%
Moreover, since $\mathrm{E}K_{5}\times _{K_{5}}R_{2}$ is a trivial vector
bundle%
\begin{equation*}
w\left( \mathrm{E}K_{5}\times _{K_{5}}R_{2}\right) =1
\end{equation*}%
Thus, the restriction diagram presented in \cite[Section 4.3.2, equations
(26) and (27)]{B-Z} implies that%
\begin{equation}
w\left( \mathrm{E}D_{8}\times _{D_{8}}R_{2}\right) =1+y.  \label{eq:SW-R_2}
\end{equation}%
On the other hand, presented in the new notation%
\begin{equation*}
w\left( \mathrm{E}H_{1}\times _{H_{1}}\mathbb{R}^{2}\right) =\left(
1+a\right) \left( 1+a+b\right) =1+b+a\left( a+b\right) .
\end{equation*}%
The $2$-dimensional real $K_{5}$-representation $\mathbb{R}^{2}$ can be
decomposed into the direct sum $\mathbb{R}^{2}\cong V_{0}\oplus V_{1}$ of
the trivial $1$-dimensional real $K_{5}$-representation $V_{0}$ and the $1$%
-dimensional real $K_{5}$-representation $V_{1}$ where the action of
generator $\sigma \in K_{5}$ is given by $\sigma \cdot v=-v$, for $v\in
V_{1} $. Then the total Stiefel--Whitney class is%
\begin{equation*}
w\left( \mathrm{E}K_{5}\times _{K_{5}}\mathbb{R}^{2}\right) =1+t_{5}.
\end{equation*}%
Again the restriction diagram \cite[Section 4.3.2, equations (26) and (27)]%
{B-Z} implies that%
\begin{equation}
w\left( \mathrm{E}D_{8}\times _{D_{8}}\mathbb{R}^{2}\right) =1+\left(
y+x\right) +w.  \label{eq:SW-R^2}
\end{equation}

\medskip

\begin{proposition}
With notation already introduced:

\begin{compactenum}[\rm(a)]
\item $\mathrm{Index}_{D_{8},\mathbb{F}_{2}}V_{n}^{2}=\langle \bar{w}_{n-1}(%
\mathrm{EO}(2)\times _{D_{8}}\mathbb{R}^{2}),\bar{w}_{n}\left( \mathrm{EO}%
(2)\times _{D_{8}}\mathbb{R}^{2}\right) \rangle\subseteq H^{*}\left(D_8,%
\mathbb{F}_2\right) $ where
\begin{equation*}
\left( 1+\bar{w}_{1}\left( \mathrm{EO}(2)\times _{D_{8}}\mathbb{R}%
^{2}\right) +\bar{w}_{2}\left( \mathrm{EO}(2)\times _{D_{8}}\mathbb{R}%
^{2}\right) +\cdots\right) \left( 1+\left( y+x\right) +w\right) =~1.
\end{equation*}

\item $\mathrm{Index}_{D_{8},\mathbb{F}_{2}}~S\left( R_{2}^{\oplus m}\right)
=\langle y^{m}\rangle .$

\item $~y^{m}\notin \langle \bar{w}_{n-1}(\mathrm{EO}(2)\times _{D_{8}}\mathbb{R%
}^{2}),\bar{w}_{n}\left( \mathrm{EO}(2)\times _{D_{8}}\mathbb{R}^{2}\right)
\rangle ~\Longrightarrow ~(n,m,2)\in \mathcal{R}_{odd,sym}^{orth}.$

\item $y^{m}\notin \langle \bar{w}_{n-1}(\mathrm{EO}(2)\times _{D_{8}}\mathbb{R%
}^{2}),\bar{w}_{n}\left( \mathrm{EO}(2)\times _{D_{8}}\mathbb{R}^{2}\right)
,x\rangle ~\Longrightarrow ~(n,m,2)\in \mathcal{R}_{odd,sym}^{orth}.$

\end{compactenum}
\end{proposition}

\begin{proof}
(a) Proposition \ref{prop:StiefelG} together with the evaluated total
Stiefel--Whitney class (\ref{eq:SW-R^2}) imply the claim.

\noindent (b) From (\ref{eq:SW-R_2}) it follows that $\mathfrak{e}(\mathrm{E}D_{8}\times _{D_{8}}R_{2})=y$ and consequently $\mathfrak{e}(\mathrm{E}D_{8}\times _{D_{8}}R_{2}^{\oplus m})=y^{m}$. Since the Fadell--Husseini index of a sphere in this case is a principal ideal generated by the Euler class~\cite[Proof of Proposition 3.11]{B-Z} the claim is proved.

\noindent (c) This is a direct consequence of the configuration test map
construction presented at the beginning of Section \ref{ram-rattray-proof}.

\noindent (d) If $y^{m}$ is not an element of the bigger ideal
\begin{equation*}
\langle \bar{w}_{n-1}(\mathrm{EO}(2)\times _{D_{8}}\mathbb{R}^{2}),\bar{w}%
_{n}\left( \mathrm{EO}(2)\times _{D_{8}}\mathbb{R}^{2}\right) ,x\rangle
\end{equation*}
it certainly can not belong to the smaller ideal
\begin{equation*}
\langle \bar{w}_{n-1}(\mathrm{EO}(2)\times _{D_{8}}\mathbb{R}^{2}),\bar{w}%
_{n}\left( \mathrm{EO}(2)\times _{D_{8}}\mathbb{R}^{2}\right) \rangle .
\end{equation*}
The statement follows from (c).
\end{proof}

\medskip

Hence, the final effort is to determine a condition on the integer $m$ such
that
\begin{equation*}
y^{m}\notin \langle \bar{w}_{n-1}(\mathrm{EO}(2)\times _{D_{8}}\mathbb{R}%
^{2}),\bar{w}_{n}\left( \mathrm{EO}(2)\times _{D_{8}}\mathbb{R}^{2}\right)
,x\rangle
\end{equation*}%
or%
\begin{equation*}
0\neq y^{m}\in \mathbb{F}_{2}[y,w]/\langle \bar{w}_{n-1},\bar{w}_{n}\rangle
\end{equation*}%
where $\left( 1+y+w\right) \left( 1+\bar{w}_{1}+\bar{w}_{2}+\cdots\right) =1$.

\medskip

If $y$ and $w$ are interpreted as the first and the second Stiefel--Whitney
class in the cohomology of the Grassmannian $G^2\left(\mathbb R^n\right)$ we can identify $%
\mathbb{F}_{2}[y,w]/\langle \bar{w}_{n-1},\bar{w}_{n}\rangle $ with $H^{\ast
}\left( G^2\left(\mathbb R^n\right);\mathbb{F}_{2}\right) $. Then our final step coincides with
the well known problem of \emph{determining the height (maximal nonzero
power) of the first Stiefel--Whitney class in the cohomology of the
Grassmannian} $G^2\left(\mathbb R^n\right)$. In~\cite[Proposition~2.6, page~525]{hil1980A} the
following statement is proved:

\begin{compactenum}[$~~$]
\item \textbf{Lemma.} Let $n\geq 2$ , and let $P(n):=2^{s}$ be the minimal power of two,
satisfying $2^{s}\geq n$. For the first Stiefel-Whitney class $w_{1}$ of the
Grassmannian $G^2(\mathbb R^n)$ holds%
\begin{equation*}
\begin{array}{ccc}
w_{1}^{2^{s}-2}\not=0 & \text{and} & w_{1}^{2^{s}-1}=0.%
\end{array}%
\end{equation*}
\end{compactenum}

\noindent Therefore,
\begin{equation*}
P(n)\geq m+2~\Longleftrightarrow ~n\geq \tfrac{1}{2}P(m+2)+1~\Longrightarrow
~(n,m,2)\in \mathcal{R}_{odd,sym}^{orth}.
\end{equation*}

\subsection{Proof of Theorem~\protect\ref{ram-2frames-proj-emb}}
\label{ram-2frames-proj-emb-proof}

Consider the Stiefel manifold $V_n^2$ with $D_8$ action on it. 
We want to know whether $V_n^2$ can be mapped $D_8$-equivariantly to $(R_2)^m\setminus \{\mathbf{0}\}$.

Denote by $\sigma_1,\sigma_2, \tau$ the generators of $D_8$, where $\sigma_1$ and $\sigma_2$ reflect the base vectors in $\mathbb R^2$, and $\tau$ transposes the base vectors. 
$R_2$ is the one-dimensional real $D_8$-representation on which $\sigma_1$ and $\sigma_2$ act antipodaly, and $\tau$ acts trivially.

Now consider an automorphism of $D_8$, defined by
\begin{eqnarray*}
\sigma'_1 &=& \sigma_1\sigma_2\tau\\
\sigma'_2 &=& \tau\\
\tau' &=& \sigma_1.
\end{eqnarray*}

Under this automorphism the representation of $D_8$ on $\mathbb R^2$ remains the same (it is sufficient to change the base $e'_1 = e_1 + e_2, e'_2=-e_1+e_2$). 
The representation $R_2$ is now given by trivial action of $\sigma'_1$ and $\sigma'_2$ and by antipodal action of $\tau'$. 
Thus, we pass to the space $X_n=V_n^2/(\sigma'_1, \sigma'_2)$ of all ordered pairs of orthogonal lines through the origin in $\mathbb R^n$. 
This space has the action of $\Z_2=(\tau')$ which permutes the lines.
We want to know whether $X$ can be mapped $\mathbb Z_2$-equivariantly to $\gamma^m\setminus\{\mathbf{0}\}$, where $\gamma$ is the unique non-trivial one-dimensional representation of $\mathbb Z_2$. 
It is well known that $X$ is homotopy equivalent to the deleted square of the projective space $\mathbb RP^{n-1}$, i.e.,
$$
X \simeq \left(\mathbb RP^{n-1}\times\mathbb RP^{n-1}\right)\setminus\Delta(\mathbb
RP^{n-1}).
$$
The existence of a $\mathbb Z_2$-equivariant map $X\to S(\gamma^m)$ is exactly the ``deleted square obstruction'' for the embedding of $\mathbb RP^{n-1}$ to $\mathbb R^m$.

The idea of considering the same automorphism of $D_8$ was used by Gonz\'{a}lez and Landweber in \cite{gonz-land2008}, where the deleted square obstruction is related to another problem of finding the symmetric topological complexity of the projective space.

\subsection{Proof of Theorem~\protect\ref{ram-3frames}}

We consider the group $G:=W_3^{(2)} = D_8\times \mathbb Z_2$. We already know that
\[
H^*(D_8, \mathbb F_2) = \mathbb F_2[x,y,w]/\langle xy\rangle,\quad H^*(\mathbb Z_2,\mathbb F_2) = \mathbb F_2[t],
\]
and therefore $H^*(G, \mathbb F_2) = \mathbb F_2[x,y,w,t]/\langle xy\rangle$ by the K\"unneth formula. The Stiefel--Whitney class of the standard $G$-representation on $\mathbb R^3$ is 
$$
w(\mathbb R^3) = (1 + x + y + w)(1+t),
$$
and the Euler class of the representation $R_3$ is 
$$
\mathfrak e(R_3) = y(t^2 + t(x+y) + w),
$$
because $\mathbb R^3(G) = \mathbb R^2(D_8)\oplus\mathbb R^1(\mathbb Z_2)$ and $R_3(G) = R_2(D_8)\oplus \mathbb R^2(D_8)\otimes \mathbb R^1(\mathbb Z_2)$ in the obvious notation. The rest of the proof proceeds in the footsteps of the proof of Theorem~\ref{ram-rattray-st}.

\subsection{Proof of Theorem~\protect\ref{ram-rattray-unit}}

\label{p-group-sec}Before proving Theorem~\ref{ram-rattray-unit} we recall some basic facts and results on the following Borsuk--Ulam type problem (consult the book~\cite{bart1993}).

\begin{compactenum}[$~~$]
\item \textbf{Problem.} Let $G$ be a finite group and $V$ its real representation such that $%
V^{G}=\{0\}$. Determine the conditions for the vector bundle
\begin{equation*}
EG\times V\rightarrow EG
\end{equation*}%
to have a $G$-equivariant nonzero section.
\end{compactenum}

\medskip

\noindent The following result for $p$-groups will be used, consult~\cite{bcp1991,bart1992,bart1993,Clap-Mar}.

\begin{compactenum}[$~~$]
\item \textbf{Lemma.} Let $G$ be a $p$-group and $V$ its real representation such that $V^{G}=\{\mathbf{0}\}$. Then the image of an equivariant map $f : EG\to V$ intersects $V^G={\mathbf{0}}$. Moreover, there exists an integer $n(G, V)$ such that for every free $G$-space $X$ is $(n-1)$-connected where
$n\ge n(G, V)$, the image of an equivariant map $f : X\to V$ meets $V^G={\mathbf{0}}$.
\end{compactenum}

\medskip

\noindent In order to prove Theorem~\ref{ram-rattray-unit} we slightly change the configuration test map construction given at the beginning of this chapter. Let us fix positive integers $k$ and $m$, and consider a collection of $m$ odd functions $f_{1},\ldots,f_{m}$. The test map in this case is the $W_{k}$-equivariant map $\upsilon :Y\rightarrow R_{k}^{\oplus m}\oplus I_{k}^{\oplus m}$ defined by
\begin{equation*}
\begin{array}{ccc}
\left( e_{1},\ldots ,e_{k}\right) & \overset{\upsilon }{\longmapsto } &
\left( \mu _{f_{r}}\left( e_{1},\ldots ,e_{k}\right) \right)
_{r=1}^{m}\oplus \left( \eta _{f_{r}}\left( e_{1},\ldots ,e_{k}\right)
\right) _{r=1}^{m}%
\end{array}%
\end{equation*}%
where $Y$ stands for the Stiefel manifold $V_{n}^{k}$ as before. If there exists a positive integer $n=n(k,m)$ such that there is no $W_{k}$-equivariant map
\begin{equation*}
Y\rightarrow \left( R_{k}^{\oplus m}\oplus I_{k}^{\oplus m}\right)
\backslash \left\{ \mathbf{0}\right\} \rightarrow S\left( R_{k}^{\oplus
m}\oplus I_{k}^{\oplus m}\right)
\end{equation*}%
then Theorem \ref{ram-rattray-unit} is proved.

\smallskip

Without loss of generality we may increase $n$ and $k$ in such a way that $k$ becomes power of $2$. This can be done since we do not need an optimal $n$ and moreover proving the theorem for bigger $k$ and fixed $n$ and $m$ yields the same result for smaller $k$. Now consider the $2$-Sylow subgroup $W_{k}^{(2)}$ of $W_{k}$. Since the $W_{k}^{(2)}$-fixed point set of the representation $R_{k}^{\oplus m}\oplus I_{k}^{\oplus m}$ is trivial, i.e., $\left( R_{k}^{\oplus m}\oplus I_{k}^{\oplus m}\right) ^{W_{k}^{(2)}}=\{\mathbf{0\}}$ the previously presented lemma implies that every map $Y\rightarrow R_{k}^{\oplus m}\oplus I_{k}^{\oplus m}$ must meet origin. Thus there cannot be any $W_{k}^{(2)}$-equivariant (and consequently $W_{k}$-equivariant) map $Y\rightarrow S\left( R_{k}^{\oplus m}\oplus I_{k}^{\oplus m}\right) $. This completes the proof of the theorem.

\subsection{Proof of Theorem \protect\ref{rattray-gen}}

Let $\lambda_1,\ldots ,\lambda_{n-k}$ be independent linear forms defining the subspace $L$ in $\mathbb{R}^{m}$. 
In this proof we take $\mathbb{R}^{k}$ to be an $\mathrm{O}(k)$-representation where the action is given by the left matrix multiplication. 
The inclusion $W_{k}\subseteq \mathrm{O}(k)$ gives to $\mathbb{R}^{k}$ also the structure of a $W_{k}$-representation.
Let us denote this $W_k$-representation by $P_k$. 
Consider the following $W_{k}$-equivariant maps
\begin{compactitem}

\item $\phi _{0}:V_{n}^{k}\rightarrow R_{k}$ given by
\begin{equation*}
\phi_0(e_1,\ldots,e_k) =(\psi(e_i),\psi(e_j))_{1\le i<j\le k},
\end{equation*}

\item $\phi _{r}:V_{n}^{k}\rightarrow P_k$, for $1\leq r\leq n-k$, given by
\begin{equation*}
\phi_r(e_1,\ldots,e_k)=\left(\lambda_r(\psi(e_1)),\ldots, \lambda_r(\psi(e_k))\right)
\end{equation*}%
for $1\leq i\leq k$.

\end{compactitem}

\noindent The sum of these maps, the $W_{k}$-equivariant map, $\phi =\phi
_{0}\oplus \phi _{1}\oplus \cdots\oplus \phi _{n-k}:V_{n}^{k}\rightarrow
R_{k}\oplus \left( P_{k}\right) ^{n-k}$ has the property that if the image
of $\phi $ meets zero in $R_{k}\oplus P_{k}^{n-k}$ then the theorem follows.
It is sufficient to show that the Euler class
\begin{equation*}
\mathfrak{e}(R_{k}\oplus P_{k}^{n-k})\in H^{\ast }(\mathrm{B}W_{k};\mathbb{F}%
_{2})
\end{equation*}%
has nonzero image in $H_{W_{k}}^{\ast }(V_{n}^{k};\mathbb{F}_{2})$, i.e.,
\begin{equation*}
\mathfrak{e}(R_{k}\oplus P_{k}^{n-k})\notin \mathrm{Index}_{\mathbb{W}_{k},%
\mathbb{F}_{2}}V_{n}^{k}.
\end{equation*}

\smallskip

Let us prove non-vanishing of the Euler class by counting zeroes of a generic map. We construct another $W_{k}$-equivariant map:

\begin{equation*}
\tau :V_{n}^{k}\rightarrow R_{k}\oplus P_{k}^{n-k}
\end{equation*}%
with the unique (up to $W_{k}$-action) non-degenerated zero. This will imply
that $\mathfrak{e}(R_{k}\oplus P_{k}^{n-k})\neq 0$ as an element of $%
H_{W_{k}}^{\ast }(V_{n}^{k};\mathbb{F}_{2})$.

\noindent Let $M=\mathbb{R}^{k}\subseteq \mathbb{R}^{n}$ be a standard
inclusion, and let $f(x,y)$ be a symmetric quadratic form, such that $%
f|_{M\times M}$ is generic. Put

\begin{equation*}
\tau _{0}(e_{1},\ldots ,e_{k})=(f(e_{i},e_{j}))_{1\leq i<j\leq k},
\end{equation*}%
and for $1\leq r\leq n-k$
\begin{equation*}
\tau _{r}(e_{1},\ldots ,e_{k})=(x_{k+r}(e_{1}),\ldots ,x_{k+r}(e_{k})),
\end{equation*}%
where $x_{k+r}$ are coordinate functions in $\mathbb{R}^{n}$. Then a unique
(up to $W_{k}$-action) basis in $M$ is mapped by $\tau $ to zero; because
the conditions $\tau_r(e_1, \ldots, e_k) = 0$ (for $1\le r\le n-k$) imply $%
e_1, \ldots, e_k\in M$ and condition $\tau_0(e_1, \ldots, e_k)=0$ implies
that $f|_{M\times M}$ is diagonal in the basis $(e_1, \ldots, e_k)$ of $M$.
This zero is non-degenerate, because the image of the differential $d\tau$
at $(e_1, \ldots, e_k)$

\begin{compactitem}
\item contains $R_k$, similar to the proof of the Rattray theorem;
\item surjects onto $P_k^{n-k}$, because in the first order approximation
the frame $(e_1+\delta_1, \ldots, e_k+\delta_k)$ is orthonormal for any $\delta_1,\ldots,\delta_k\in M^\perp$.
\end{compactitem}

\noindent Thus $0\neq \mathfrak{e}(R_{k}\oplus P_{k}^{n-k})\in
H_{W_{k}}^{\ast }(V_{n}^{k};\mathbb{F}_{2})$ and the proof is complete.

\section{Proof of Makeev type results}

\subsection{Proof of Theorem~\protect\ref{Makeev-general}}

Makeev type results will be considered via the classical configuration space~/~test map scheme used for mass partition problems by hyperplanes, consult~\cite{zvm2006} or~\cite{B-Z} for more details. We consider two different configuration spaces depending whether we require configuration of orthogonal hyperplanes or not.

\smallskip

Let $\mathbb{R}^{n}$ be embedded in $\mathbb{R}^{n+1}$ by $(x_{1},\ldots,x_{n})\longmapsto (x_{1},\ldots,x_{n},1)$. Every oriented affine hyperplane $H$ in $\mathbb{R}^{n}$ determines a unique oriented hyperplane $H^{\prime }$ through the origin in $\mathbb{R}^{n+1}$ by $H^{\prime }\cap\mathbb{R}^{n}=H$. Converse is also true if the hyperplane $x_{n+1}=0$ is excluded. Any oriented hyperplane $H$ in $\mathbb{R}^{n+1}$ passing through the origin is uniquely determined by the unit vector $v\in S^{d}$ pointing
inside the halfspace $H^{+}$. Such a hyperplane we denote also by $H_{v}$. Notice that $H_{-v}^{-}=H_{v}^{+}$. Thus, the space of all oriented affine hyperplanes in $\mathbb{R}^{n}$ (including two hyperplanes at ``infinity'') can be considered to be the sphere $S^{n}$. 
The first configuration space we consider is
\begin{equation*}
\begin{array}{lllll}
X & = & \left( S^{n}\right) ^{k} & = & \text{the space of all collections of
}k\text{ oriented affine hyperplanes in }\mathbb{R}^{n}\text{.}%
\end{array}%
\end{equation*}

\smallskip

\noindent Let $\mu$ be an absolutely continuous probabilistic measure on $\mathbb{R}^{n}$ with connected support. Then the second configuration space $Y_{\mu }=V_{n}^{k}$ is shaped by $\mu$ in the following way: every orthonormal $k$-frame $\left(e_{1},\ldots,e_{k}\right) \in V_{n}^{k}$ determines a unique collection of $k$ oriented affine hyperplanes $\left( H_{1},\ldots,H_{k}\right) $ in $\mathbb{R}^n$ with the property that $e_{i}\perp H_{i}$ and $\mu \left(H_{i}^{+}\right) =\mu \left( H_{i}^{-}\right) $ for all $1\leq i\leq k$. This is because for every given direction $e_i$ there is a unique hyperplane orthogonal to $e_i$ that partitions $\mu$ into equal halves. In case $\mu$ has disconnected support, we may approximate $\mu$ by a sequence of measures with connected support, prove the theorem in this case, and then go to the limit  using the compactness of the following space: for a given $0<\varepsilon<1$ consider the space of hyperplanes $H$ that partition $\mu$ into parts $H^+, H^-$ with difference $|\mu(H^+) - \mu(H^-)|\le\epsilon$. 

\smallskip

\noindent The group $W_{k}=(\mathbb{Z}_{2})^{k}\rtimes \Sigma _{k}\subset\mathrm{O}(k)$ acts on both configuration spaces $X$ and $Y$ in the same way as in Section~\ref{Section:Proof of Rattray}.

\smallskip

Before defining the test maps let us introduce a particular $W_{k}$ and $(\mathbb{Z}_{2})^{k}$-representation on the vector space $\mathbb{R}^{2^{k}}$and study its structure. If we assume that the coordinate functions $x_{\left( a_{1},\ldots ,a_{k}\right) }$ on $\mathbb{R}^{2^{k}}$ are indexed by the elements $\left( a_{1},\ldots ,a_{k}\right) $ of the group $(\mathbb{Z}_{2})^{k}$, then the $W_{k}$-action we consider is given by
\begin{equation*}
\left( \left( b_{1},\ldots,b_{k}\right) \rtimes \pi \right) \cdot x_{\left(
a_{1},\ldots ,a_{k}\right) }=x_{\left( b_{1}a_{\pi ^{-1}\left( 1\right)
},\ldots ,b_{k}a_{\pi ^{-1}(k)}\right) }
\end{equation*}%
where $\left( b_{1},\ldots,b_{k}\right) \in (\mathbb{Z}_{2})^{k}$ and $\pi \in
\Sigma _{k}$. The inclusion $(\mathbb{Z}_{2})^{k}\subset W_{k}$ induces also
the structure of $(\mathbb{Z}_{2})^{k}$-representation on $\mathbb{R}%
^{2^{k}} $.

\smallskip

All real irreducible representations of the group $(\mathbb{Z}_{2})^{k} $ are all $1$-dimensional. 
They are completely determined by characters $\chi :(\mathbb{Z}_{2})^{k}\rightarrow \mathbb{Z}_{2}$. 
For $(a_{1},\ldots,a_{k})\in (\mathbb{Z}_{2})^{k}=\{+1,-1\}^{2^{k}}$, let 
\[
V_{a_{1}\ldots a_{k}}=\mathrm{span}\{v_{a_{1},\ldots ,a_{k}}\}\subset\mathbb{R}^{2^{k}}
\]
denotes the $1$-dimensional representation given by
\begin{equation*}
\varepsilon _{i}\cdot v_{a_{1}\ldots a_{k}}=a_{i}~v_{a_{1}\ldots a_{k}}.
\end{equation*}%
Then there is a decomposition of the real $(\mathbb{Z}_{2})^{k}$-representation
\begin{equation*}
\mathbb{R}^{2^{k}}\cong \sum\limits_{a_{1},\ldots ,a_{k}\in (\mathbb{Z}%
_{2})^{k}}V_{a_{1}\ldots a_{k}}\cong V_{+\cdots +}\oplus
\sum\limits_{a_{1},\ldots ,a_{k}\in (\mathbb{Z}_{2})^{k}\backslash
\{+\cdots +\}}V_{a_{1},\ldots ,a_{k}}.
\end{equation*}%
Observe that $V_{+\cdots +}$ is the trivial $1$-dimensional real $(\mathbb{Z}_{2})^{k}$-representation. 
In order to simplify further notation let us define for $1\leq i\leq j\leq k$ the following $(\mathbb{Z}_{2})^{k}$-representation
\begin{equation*}
S_{ij}=\sum\limits_{\substack{ a_{1},\ldots ,a_{k}\in (\mathbb{Z}%
_{2})^{k}\backslash \{+\cdots+\}  \\ i\leq s(a_{1},\ldots ,a_{k})\leq j}}%
V_{a_{1}\ldots a_{k}}
\end{equation*}%
where $s(a_{1},\ldots ,a_{k})$ denotes the number of $-1$ in the sequence $(a_{1},\ldots ,a_{k})$.

\smallskip

Let $\mu _{1},\ldots ,\mu _{m}$ be a collection of $m$ absolutely continuous
probabilistic measures on $\mathbb{R}^{n}$. The test maps we consider
\begin{equation*}
\begin{array}{ccc}
\tau :X\rightarrow S_{1l}^{\oplus m} & \text{and} & \tau ^{orth}:Y_{\mu
_{1}}\rightarrow S_{1l}^{\oplus m}%
\end{array}%
\end{equation*}%
are defined by%
\begin{equation*}
\begin{array}{ccc}
\left( v_{1},\ldots ,v_{k}\right) & \overset{\tau }{\longmapsto } & \left(
\left( \mu _{i}(H_{v_{1}}^{a_{1}}\cap\cdots\cap H_{v_{k}}^{a_{k}})-\tfrac{1}{%
2^{k}}\mu _{i}(\mathbb{R}^{d})\right) _{(a_{1},\ldots ,a_{k})\in (\mathbb{Z}%
_{2})^{k}}\right) _{i\in \{1,\ldots ,m\}} \\
\left( e_{1},\ldots ,e_{k}\right) & \overset{\tau ^{orth}}{\longmapsto } &
\left( \left( \mu _{i}(H_{e_{1}}^{a_{1}}\cap\cdots\cap H_{e_{k}}^{a_{k}})-%
\tfrac{1}{2^{k}}\mu _{i}(\mathbb{R}^{d})\right) _{(a_{1},\ldots ,a_{k})\in (
\mathbb{Z}_{2})^{k}}\right) _{i\in \{1,\ldots ,m\}}%
\end{array}%
\end{equation*}%
for $\left( v_{1},\ldots,v_{k}\right) \in X$ and $\left( e_{1},\ldots ,e_{k}\right)
\in Y_{\mu _{1}}$. Since the configuration space $Y_{\mu_{1}}$ is chosen in
such a way that each hyperplane equipartitions the measure $\mu_{1}$ the test
map $\tau ^{orth}$ factors
\begin{equation*}
Y_{\mu _{1}}\overset{\rho }{\longrightarrow }S_{2l}\oplus S_{1l}^{\oplus
\left( m-1\right) }\overset{\iota }{\longrightarrow }S_{1l}^{\oplus m}
\end{equation*}%
so that $\tau ^{orth}=\iota \circ \rho $ and $\iota $ is induced by the
inclusion $S_{2l}\rightarrow S_{1l}$.

\smallskip

\noindent All test maps $\tau $, $\tau ^{orth}$ and $\rho $ are $W_{k}$-equivariant maps, when the introduced actions on the spaces are assumed. The key property of these test maps is that: If for every collection $\mu_{1},\ldots ,\mu _{m}$ of $m$ absolutely continuous probabilistic measures on $\mathbb{R}^{n}$

\begin{compactitem}
\item if $\left\{ \mathbf{0}\in S_{1l}^{\oplus m}\right\} \in \tau \left(
X\right) $, then $\left( n,m,k,l\right) \in \mathcal{M}$,

\item if $\left\{ \mathbf{0}\in S_{2l}\oplus S_{1l}^{\oplus \left( m-1\right)
}\right\} \in \rho \left( Y_{\mu _{1}}\right) $, then $\left( n,m,k,l\right)
\in \mathcal{M}^{orth}$.
\end{compactitem}

\noindent Using the contraposition we get that

\begin{compactitem}[$~$]
\item
\begin{tabular}{llll}
$\bullet $ & $\left( n,m,k,l\right) \notin \mathcal{M}$ & $\Longrightarrow $
& there exists a collection of $m$ absolutely continuous probabilistic  \\
&  &  & measures on $\R^{n}$ such that $\left\{ \mathbf{0}\in S_{1l}^{\oplus m}\right\} \notin
\tau \left( X\right) $ \\
&  & $\Longrightarrow $ & there exists a $W_{k}$-equivariant map \\
&  &  & $X=(S^n)^k\rightarrow S_{1l}^{\oplus m}\backslash \left\{ \mathbf{0}\right\}
\rightarrow S\left( S_{1l}^{\oplus m}\right) ,$%
\end{tabular}

\item
\begin{tabular}{llll}
$\bullet $ & $\left( n,m,k,l\right) \in \mathcal{M}^{orth}$ & $%
\Longrightarrow $ & there exists a collection of $m$ absolutely continuous
probabilistic \\
&  &  & measures on $\R^{n}$  such that $\left\{ \mathbf{0}\in S_{2l}\oplus S_{1l}^{\oplus \left(
m-1\right) }\right\} \notin \rho \left( Y_{\mu _{1}}\right) $ \\
&  & $\Longrightarrow $ & there exists a $W_{k}$-equivariant map \\
&  &  & $Y_{\mu_1}=V_n^{k}\rightarrow S_{2l}\oplus S_{1l}^{\oplus \left( m-1\right)
}\backslash \left\{ \mathbf{0}\right\} \rightarrow S\left( S_{2l}\oplus
S_{1l}^{\oplus \left( m-1\right) }\right) .$%
\end{tabular}
\end{compactitem}

\noindent This implies that

\begin{compactitem}

\item if there is no $W_{k}$-equivariant map $X=(S^n)^k\rightarrow S\left( S_{1l}^{\oplus m}\right)$,
then $\left( n,m,k,l\right) \in \mathcal{M}$,

\item if there is no $W_{k}$-equivariant map $Y_{\mu_1}=V_n^{k}\rightarrow S\left( S_{2l}\oplus
S_{1l}^{\oplus \left( m-1\right) }\right)$, then $\left( n,m,k,l\right) \in \mathcal{M}^{orth}.$
\end{compactitem}

\noindent Therefore, by proving the following statement we conclude the proof of Theorem \ref{Makeev-general}.

\begin{proposition}
~~
\begin{compactenum}[\rm(a)]

\item If
$$
\prod_{\substack{ s_{1},\ldots ,s_{k}\in \mathbb{Z}_{2} \\ 1\leq
s_{1}+\ldots +s_{k}\leq l}}(s_{1}t_{1}+s_{2}t_{2}+\dots
+s_{k}t_{k})^{m}\notin \langle t_{1}^{n+1},\ldots ,t_{k}^{n+1}\rangle
$$
then there is no $W_{k}$-equivariant map $X=(S^n)^k\rightarrow S\left( S_{1l}^{\oplus m}\right)$,

\item If
$$
\frac{1}{t_{1}\cdots t_{k}}\prod_{\substack{ s_{1},\ldots ,s_{k}\in \mathbb{%
Z}_{2} \\ 1\leq s_{1}+\ldots +s_{k}\leq l}}(s_{1}t_{1}+s_{2}t_{2}+\dots
+s_{k}t_{k})^{m}\notin \langle \bar{w}_{n-k+1},\ldots ,\bar{w}_{n}\rangle
$$
then there is no $W_{k}$-equivariant map $Y_{\mu_1}=V_n^{k}\rightarrow S\left( S_{2l}\oplus
S_{1l}^{\oplus \left( m-1\right) }\right)$.
\end{compactenum}
\end{proposition}

\begin{proof}
Both statements follow from the Fadell--Husseini index computations:%
\begin{equation*}
\begin{array}{llll}
\mathrm{Index}_{\mathbb{Z}_{2}^{k},\mathbb{F}_{2}}~\left( S^{n}\right) ^{k}
& = & \langle t_{1}^{n+1},\ldots ,t_{k}^{n+1}\rangle  , \\
\mathrm{Index}_{\mathbb{Z}_{2}^{k},\mathbb{F}_{2}}~S_{1l}^{\oplus m} & = &
\langle \prod_{\substack{ s_{1},\ldots ,s_{k}\in \mathbb{Z}_{2} \\ 1\leq
s_{1}+\ldots +s_{k}\leq l}}(s_{1}t_{1}+s_{2}t_{2}+\dots
+s_{k}t_{k})^{m}\rangle  , \\
\mathrm{Index}_{\mathbb{Z}_{2}^{k},\mathbb{F}_{2}}~V_{n}^{k} & = & \langle
\bar{\omega}_{n-k+1},\ldots ,\bar{\omega}_{n}\rangle  , \\
\mathrm{Index}_{\mathbb{Z}_{2}^{k},\mathbb{F}_{2}}S_{2l}\oplus
S_{1l}^{\oplus \left( m-1\right) } & = & \langle \frac{1}{t_{1}\cdots t_{k}}\prod
_{\substack{ s_{1},\ldots ,s_{k}\in \mathbb{Z}_{2} \\ 1\leq s_{1}+\ldots
+s_{k}\leq l}}(s_{1}t_{1}+s_{2}t_{2}+\dots +s_{k}t_{k})^{m}\rangle  ,%
\end{array}%
\end{equation*}%
and its basic property that if there is a $G$-equivariant map $X\rightarrow Y
$ then $\mathrm{Index}_{G,\ast }X\supseteq \mathrm{Index}_{G,\ast }Y$.
\end{proof}

\subsection{Proof of Theorem~\ref{proj-ham-sandwich}}

Let us lift the measures to $S^{n-1}\subseteq \mathbb R^n$; we obtain $m+1$ centrally symmetric measures on the sphere. It is sufficient to find a pair of oriented hyperplanes through the origin $H_1, H_2$ such that for every $i=0,1,\ldots, m$
$$
\mu_i(H_1^+\cap H_2^+) = \mu_i(H_1^+\cap H_2^-) = \mu_i(H_1^-\cap H_2^+) = \mu_i(H_1^-\cap H_2^-).
$$
Since the conditions $\mu_i(H_1^+\cap H_2^+) = \mu_i(H_1^-\cap H_2^-)$ and $\mu_i(H_1^+\cap H_2^-) = \mu_i(H_1^-\cap H_2^+)$ hold always (because of the central symmetry), we may select the components of the test map to be
$$
f_i(H_1, H_2) = \mu_i(H_1^+\cap H_2^+) - \mu_i(H_1^+\cap H_2^-) - \mu_i(H_1^-\cap H_2^+) + \mu_i(H_1^-\cap H_2^-)
$$

The rest of the proof would follow directly from the proof of Theorem~\ref{ram-2frames-proj-emb} (see Section~\ref{ram-2frames-proj-emb-proof}), if we had $m$ measures. We are going to provide an additional argument to partition $m+1$ measures.

Take the measure $\mu_0$ and assume that its support equals $S^{n-1}$. Any measure can be approximated by such a measure, and the standard compactness argument (the configuration space of all pairs $(H_1, H_2)$ is compact) extends the solution to arbitrary measures. We are going to show the following:

\begin{proposition}
If the support of $\mu_0$ is the whole $S^{n-1}$, then the configuration space $X$ of pairs $(H_1, H_2)$ that equipartition $\mu_0$ (i.e. $f_0(H_1, H_2)=0$) is $D_8$-equivariantly homeomorphic to $V_n^2$.
\end{proposition}

\begin{proof}
Take an orthogonal $2$-frame $(e_1, e_2)$. Denote the orthogonal complement of $(e_1, e_2)$ by $L^\perp(e_1, e_2)$, and denote the reflections 
$$
\sigma_1 : x\mapsto x - 2 (x,e_1)e_1,\quad \sigma_2 : x\mapsto x - 2 (x,e_2)e_2
$$
Note that the hyperplane $H_1$ is uniquely defined by the following conditions:

\begin{compactitem}
\item $H_1\supseteq L^\perp(e_1, e_2)$,
\item $e_1,e_2\in H_1^+$,
\item $H_2 = \sigma_1(H_1) = -\sigma_2(H_1)$,
\item $f_0(H_1, H_2) = 0$.
\end{compactitem}

The dependence of $H_1$ on $(e_1,e_2)\in V_n^2$ is continuous, and therefore we obtain a homeomorphism between $X$ and $V_n^2$, if the action of $D_8$ on $V_n^2$ is chosen properly.
\end{proof}

Now we continue the proof of Theorem~\ref{proj-ham-sandwich}. The functions $f_1,\ldots, f_m$ may be considered as functions on $V_n^2$. If we consider the group $\mathbb Z_2\times \mathbb Z_2\subset D_8$, generated by $\sigma_1, \sigma_2$, then the functions $f_i$ are invariant under this group action. Therefore they define the $\mathbb Z_2=D_8/(\mathbb Z_2\times \mathbb Z_2)$-equivariant map
$$
\tilde f : V_n^2/(\mathbb Z_2\times \mathbb Z_2)\simeq \left(\mathbb RP^{n-1}\times \mathbb RP^{n-1}\right)\setminus \Delta(\mathbb RP^{n-1}) \to \mathbb R^m,
$$
where the action on $\left(\mathbb RP^{n-1}\times \mathbb RP^{n-1}\right)\setminus \Delta(\mathbb RP^{n-1})$ is given by interchanging factors in the product while the action on $\mathbb R^m$ is antipodal. 
This map must have a zero, because the ``deleted square obstruction'' guarantees the existence of a zero.

\end{document}